\documentclass[12pt]{amsart}
\usepackage{a4wide}
\usepackage{graphicx}
\usepackage{color}
\usepackage{amsmath}
\usepackage{enumitem}

\usepackage[mathscr]{eucal}
\usepackage{amssymb}
\usepackage{amsfonts}
\usepackage{amscd}
\usepackage{amsthm}
\usepackage{latexsym}

\allowdisplaybreaks

\let\pa\partial

\let\var\varepsilon
\newcommand{\N}{{\mathbb N}}
\newcommand{\R}{{\mathbb R}}

\usepackage{xcolor}
\newtheorem{thm}{Theorem}

\newtheorem{rem}[thm]{Remark}

\newcommand{\dint}{\displaystyle\int}

\usepackage{mathrsfs,a4wide}
\usepackage{stackrel}

\begin{document}

\title[Reaction-diffusion equations with aggregation]{On a class of reaction-diffusion equations with aggregation}

\author[L. Chen]{Li Chen}
\address{Universit\"at Mannheim, 68131, Mannheim}
\email{chen@math.uni-mannheim.de}

\author[L. Desvillettes]{Laurent Desvillettes}
\address{Universit\'{e} Paris Diderot, Sorbonne Paris Cit\'{e}, Institut de Math\'{e}matiques de Jussieu-Paris Rive Gauche, UMR 7586, CNRS, Sorbonne Universit\'{e}s, UPMC Univ. Paris 06, F-75013, Paris, France}
\email{desvillettes@math.univ-paris-diderot.fr}

\author[E. Latos]{Evangelos Latos}
\address{University of Graz
Institute for Mathematics and Scientific Computing, A-8010 Graz, Heinrichstr. 36, Austria}
\email{evangelos.latos@uni-graz.at}

\date{\today}

\thanks{This work was supported by DFG Project CH 955/3-1. The authors have been partially supported by
the French ``ANR blanche'' project Kibord, grant ANR-13-BS01-0004, and
from the Universit\'e Sorbonne Paris Cit\'e, in the framework of the
``Investissements d'Avenir'', grant ANR-11-IDEX-0005.}

\begin{abstract}
In this paper, global-in-time existence and blow up results are shown for a reaction-diffusion equation appearing in the theory of aggregation phenomena (including chemotaxis). Properties of the corresponding steady-state problem are also presented. Moreover, the stability around constant equilibria and the non-existence of non-constant solutions are studied in certain cases.
\end{abstract}


\subjclass[2000]{}

\maketitle

\section{Introduction}


 We consider the following initial boundary value problem
\begin{align}\label{cdp0}
  \left\{\begin{array}{ll}
  \partial_t u
  =  \Delta [(a-b u)u]+(c-d u)u
 & \text{in } \Omega\times(0,T),  \\
        \mathcal{B}[u]=0
 & \text{on } \partial\Omega \times(0,T),\\
u(x,0)=u_0(x)\geq 0
        \end{array}
\right.
\end{align}
where $a,b>0$ and $c,d\in\mathbb{R}$. Here, $\Omega$ is an open bounded domain in $\mathbb{R}^n$, and $ \mathcal{B}[u]$ denotes a boundary operator of Neumann or Dirichlet type, 
i.e.
\begin{eqnarray*}
	u|_{\partial\Omega}=0, \quad  \mbox{ or } \quad (a-2bu)\nabla u\cdot\gamma|_{\partial\Omega}=0.
\end{eqnarray*}
where $\gamma$ is the outer unit normal vector of $\partial\Omega$. For the sake of simplicity, we take $|\Omega|=1$.

One of the motivations to study such an equation comes from  the structure similarities that exist with the parabolic-elliptic Keller-Segel models for chemotaxis, i.e.
$$
\partial_tu -\nabla\cdot (\nabla u -u\nabla V * u)=0,
$$
where $V$ is the fundamental solution of Poisson equation (or some other given potential in the case of general diffusion aggregation equations).
If formally the interaction potential $V$ is replaced by a the Dirac mass $\delta_0$, then the above equation is reduced to
$$
\partial_tu -\nabla\cdot (\nabla u -u\nabla u)=0.
$$

In \cite{cgk18}, the authors propose a  microscopic particle model which converges at the formal level towards such a PDE.

This  microscopic particle model, corresponding to \eqref{cdp0} with $c=d=0$ and
$\Omega = \R^n$, is the following:

\begin{eqnarray*}
&&d X^i_t= \sqrt{2a}\,dB^i_t +\frac{1}{N}\sum_{j\neq i} \nabla V_\varepsilon(|X^i_t-X^j_t|)dt,\\
&&X^i_0=\xi_i \quad \mbox{ i.i.d. random variables with distribution } u_0,
\end{eqnarray*}
where $V_\varepsilon(x)=\varepsilon^{-n}V(x/\varepsilon)$ for $\varepsilon>0$, $\dint_{\mathbb{R}^n}V(x)dx=2b$, and $B^i_t$ are i.i.d. Brownian motions.

It is well known (for example in \cite{Sni,Oel} that under suitable assumptions on $V$, the particle model converges (when $N$ goes to infinity) to the following intermediate (with fixed $\var>0$) nonlocal problem:
\begin{eqnarray*}
&&d \bar X^i_t= \sqrt{2a}\,dB^i_t+\int_{\mathbb{R}^n} \nabla V_\varepsilon(|\bar X^i_t-y|) u^\varepsilon(y,t)dydt,\\
&&\bar X^i_0=\xi_i \quad \mbox{ i.i.d. random variables with distribution } u_0,
\end{eqnarray*}
where $u^\varepsilon$ is the distribution of the i.i.d. random processes $\bar X^i_t$ at time $t$. By It$\hat o$'s formula one can obtain the following nonlocal partial differential equation for $u^\varepsilon$:
\begin{equation}\label{star1}
	\partial_tu^\varepsilon -\nabla\cdot (a\nabla u^\varepsilon -u^\varepsilon\nabla V_\varepsilon * u^\varepsilon)=0.
\end{equation}
Furthermore, in the parabolic regime, i.e. when $0\leq u_0<\frac{a}{2b},$ it is proved in \cite{cgk18} that the limit of $u^\varepsilon$ satisfies eq. (\ref{cdp0}) with $c=d=0$:
\begin{equation}\label{star2}
	\partial _t u- \Delta (a\, u - b\, u^2) =0.
\end{equation}

The physical meaning of the unknown $u$ is that of a concentration, therefore one considers only nonnegative solutions corresponding to nonnegative initial data.

Furthermore (like in the case of Keller-Segel system), problem \eqref{cdp0} with $c=d=0$ (and the homogeneous Neumann boundary condition) possesses the following entropy structure:
\begin{equation}
\frac{d}{dt}\mathcal{E}(t)
:=
\frac{d}{dt}\int_\Omega\left(
au(\log u-1)-bu^2
\right)\;dx
=
 -\int_\Omega \frac{1}{u}(a-2bu)^2|\nabla u|^2\;dx \le 0.
\end{equation}
This entropy is a combination of a positive part from the diffusion and a negative one from the aggregation. It needs to be pointed out that here the aggregation phenomenon is much stronger than the one appearing in Keller-Segel systems because of the singular potential that appeared
 in \eqref{star1},\eqref{star2}.
\medskip

As for the reaction term, it is considered to be of logistic (mono-stable) type so that (when $d>0$) a significant dampening effect is exercised on the density $u$ at those points where $u$ becomes large.
\medskip

The arrangement of the paper is the following.
In section 2, global existence and uniqueness of classical solutions are obtained for initial data such that parabolicity is expected to hold.  The rest of the paper concerns cases in which parabolicity is expected to be lost at some point, so that blowup may happen. Considerations on the possible steady states and their stability as well as direct estimates of blowup are presented. Section 3 is devoted to the study of the steady states. The non-existence of non-trivial steady states is proved via Pohozaev's type arguments. Furthermore, the linear stability of constant steady states is investigated. Finally, in section 4,  blow up (in finite time) results  are presented. Two different procedures  are carried out: Kaplan's method is used for the problem with Dirichlet boundary condition on one hand, and the concavity method is used for the problem with Neumann boundary condition on the other hand. In the end,  we present an annex where blow up is directly observed in a class of explicit solutions linked to Barenblatt profiles, and we draw conclusions in a final section.

\section{Global existence}

In this section, the global existence and uniqueness of a solution is obtained thanks to Leray-Schauder fixed point theorem, under the condition that the initial datum belongs to the parabolic region.
\medskip

Observing that
$$ \Delta [(a-b u)u]
  =  -b\Delta \left[\left(u-\frac{a}{2b}\right)^2 \right], $$
and using the notation $v := u-\frac{a}{2b}$, \eqref{cdp0} can be rewritten as
\begin{align}\label{cdp1}
  \left\{\begin{array}{ll}
  \partial_t v
  =  -b\Delta (v^2) +(c-d\frac{a}{2b}-dv)(v+\frac{a}{2b})
 & \text{in } \Omega\times(0,T),  \\
         \mathcal{B}[v+\frac{a}{2b}]=0
 & \text{on } \partial\Omega \times(0,T),\\
        v(x,0)=v_0(x).
        \end{array}
\right.
\end{align}
It can be expected that global existence holds in the case when the parabolicity can be kept in the evolution (that is, when for all time $v<0$, or equivalently $u < \frac{a}{2b}$). At the same time, the logistic term $u\,(c-du)$ and the expected nonnegativity of $u$ imply that the estimate $0\le u \le \frac{c}{d}$ should also hold. Therefore, a natural sufficient condition for getting global existence for eq. (\ref{cdp0}) is $0 \le u_0 < \frac{a}{2b}$, together with $\frac{c}{d} < \frac{a}{2b}$. The theorem below states a precise result in this direction:
\begin{thm}[Global Existence and uniqueness]\label{GlexTh}
	Let $\Omega$ be a smooth bounded domain in $\R^n$.
Assume $a,b>0$ and $\frac{a}{2b}>\frac{c}{d}$.
Let also $0\leq u_0\in C^\alpha(\overline{\Omega}),\ \alpha\in(0,1)$ and
	\begin{equation} \label{assumu0}
	\max_{x\in{\overline{\Omega}}} u_0(x)< \frac{a}{2b},
	\end{equation}
	with compatibility condition $\mathcal{B}[u_0]=0$. Then
 problem \eqref{cdp0}, together with homogeneous Neumann or homogeneous Dirichlet boundary condition
 $\mathcal{B}[u]=0$ has a unique global-in-time classical solution. In addition, it holds that
\begin{equation}
	0\leq u(x,t)<\frac{a}{2b}, \quad \forall x\in \overline{\Omega}, t\geq 0.
\end{equation}
\end{thm}

\begin{proof} We first observe that we can take $\varepsilon_0>0$ small enough in such a way that $\max_{x\in{\overline{\Omega}}} u_0(x)\leq \frac{a}{2b}-\varepsilon_0$ and $\frac{c}{d} \leq  \frac{a}{2b}-\varepsilon_0$. Then we will prove the existence and uniqueness of a solution $u$ of the problem, which satisfies the estimate 
 $$
 0\leq u(x,t)\leq \frac{a}{2b}-\varepsilon_0, \quad \forall x\in \overline{\Omega}, t\geq 0.
 $$
 For any fixed $T>0$, we will use the Leray-Schauder fixed point theorem to prove the existence. Let
$$
X=\{w\in C^{\alpha,\frac{\alpha}{2}}(\bar\Omega\times [0,T]): 0\leq w(x,t)\leq \frac{a}{2b}-\varepsilon_0,\forall x\in\overline{\Omega},t\geq0\}.
$$
We define an operator in the following way: for given $w\in X$ and $\sigma\in[0,1]$, let $u :=\mathcal T (w,\sigma)$ be the $C^{2+\alpha,1+\frac{\alpha}{2}}(\bar{\Omega}\times [0,T])$  solution (see \cite{LSU}, Chapter V, Theorem 7.4 for the existence and uniqueness of the solution) of the following problem
\begin{align}\label{defT}
 \left\{ \begin{array}{ll}
  \partial_t u -(1-\sigma)\Delta u-\sigma(a-2bw)\Delta u+2\sigma b|\nabla u|^2
  =\sigma u (c-du)
 & \text{in } \Omega\times(0,T),  \\
         \mathcal{B}[u]=0
 & \text{on } \partial\Omega \times(0,T),\\
        u(x,0)=\sigma u_0(x).
        \end{array} \right.
\end{align}
In order to build up the map $\mathcal T$, we have to show that $0\leq u(x,t)\leq  \frac{a}{2b}-\varepsilon_0$, $\forall (x,t)\in\bar{\Omega}\times[0,T]$.

For $\sigma=0$, it is obvious that $0\leq u\leq \frac{a}{2b}-\varepsilon_0$, since in that case $u$ satisfies the heat equation.
\par
For $\sigma\in(0,1]$, we first prove that $u\geq 0$.

Let $\varepsilon>0$ be small and $u^\varepsilon$ be the solution of
$$
\left\{  \begin{array}{ll}
  \partial_t u^\varepsilon -(1-\sigma)\Delta u^\varepsilon-\sigma(a-2bw)\Delta u^\varepsilon+2\sigma b|\nabla u^\varepsilon|^2
  =\sigma u^\varepsilon (c-du^\varepsilon)+\varepsilon
 & \text{in } \Omega\times(0,T),  \\
         \mathcal{B}[u^\varepsilon]=0
 & \text{on } \partial\Omega \times(0,T),\\
        u^\varepsilon(x,0)=\sigma u_0(x).
        \end{array} \right.
$$
The solution $u^\varepsilon\in C^{2+\alpha,1+\frac{\alpha}{2}}(\overline{\Omega}\times[0,T])$ possesses a uniform in $\varepsilon$ estimate $\|u^\varepsilon\|_{2+\alpha,1+\frac{\alpha}{2}}\leq C$, see \cite{LSU}.
With $u_0\geq 0$, if $\min_{\bar\Omega\times [0,T]}u^\varepsilon(x,t)< 0$, then $\exists (x_1,t_1)\in {\overline{\Omega}}\times (0,T]$ such that
\begin{equation}\label{min}
0 = u^\varepsilon(x_1,t_1)=\min_{\bar\Omega} u^\varepsilon(x,t_1)\quad \mbox{ with } \quad \partial_tu^\varepsilon(x_1,t_1)\leq 0.
\end{equation}
More precisely we take here $t_1\geq 0$ as the last time before the solution takes some negative value. If $x_1\in \Omega$, we have $\nabla u^\varepsilon(x_1,t_1)=0$. If $x_1\in\partial\Omega$, in the case of the homogeneous Neumann boundary condition, we also have $\nabla u^\varepsilon(x_1,t_1)=0$.
Then we get (still in the case of Neumann boundary condition).
$$
0\geq \Big(\partial_t u^\varepsilon -(1-\sigma)\Delta u^\varepsilon-\sigma(a-2bw)\Delta u^\varepsilon+2\sigma b|\nabla u^\varepsilon|^2\Big)\Big|_{(x_1,t_1)}=\sigma u^\varepsilon(c-du^\varepsilon)|_{(x_1,t_1)}+\varepsilon >0,
$$
which is a contradiction. Therefore, $u^\varepsilon\geq 0$.

If $x_1\in\partial\Omega$ (and $u^\varepsilon(x,t_1) >0$ for all $x \in \Omega$), in the case of the homogeneous Dirichlet boundary condition, we can prove (see the sequel of the proof) that there exists a sequence $x_n\in\Omega$ satisfying $x_n\rightarrow x_1$ and such that 
\begin{equation}\label{minnew}
	\lim_{n\rightarrow\infty}\frac{u^\varepsilon(x_n,t_1)-u^\varepsilon(x_1,t_1)}{|x_n-x_1|}=\lim_{n\rightarrow\infty}\frac{u^\varepsilon(x_n,t_1)}{|x_n-x_1|}=0.
\end{equation}
This limit, together with the fact that $u^\varepsilon$ is smooth and that the tangential derivative of $u^\varepsilon$ vanishes because of the homogeneous Dirichlet boundary condition, shows that $\nabla u^\varepsilon|_{(x_1,t_1)}=0$. Thus we can follow the same argument as above. 
\\
In order to show the limit (\ref{minnew}), we consider
$t_n=t_1+\frac{1}{n}$ a sequence that converges to $t_1$ and $x_n$ ($n\ge 2$) one of the minimal points of $u^\varepsilon(x,t_n)$ such that $x_n\rightarrow x_1$ (note that if several points $x_1 \in \partial\Omega$  satisfy (\ref{min}), one at least can be selected in such a way that the construction above makes sense). Then $u^\varepsilon(x_n,t_n)\rightarrow u^\varepsilon(x_1,t_1)$ because of the continuity of $u^\varepsilon$. We get therefore
$$
u^\varepsilon(x_n,t_1)>0\mbox{ and } u^\varepsilon(x_n,t_n)<0. 
$$
Then, the mean value theorem implies that there exists a sequence $t_n^*\in (t_1,t_n)$ such that $u^\varepsilon(x_n,t_n^*)=0$. Therefore,
\begin{eqnarray*}
&&\lim_{n\rightarrow\infty}\frac{u^\varepsilon(x_n,t_1)-u^\varepsilon(x_1,t_1)}{|x_n-x_1|}=\lim_{n\rightarrow\infty}\frac{u^\varepsilon(x_n,t_1)-u^\varepsilon(x_n,t_n^*)}{|x_n-x_1|}\\
&=&\lim_{n\rightarrow\infty}\frac{1}{|x_n-x_1|}\frac{\partial u^\varepsilon}{\partial t}(x_n,t_n^{**})(t_1-t_n^*)\\
&=& \lim_{n\rightarrow\infty}\frac{1}{|x_n-x_1|}\Big(\frac{\partial u^\varepsilon}{\partial t}(x_n,t_n^{**})-\frac{\partial u^\varepsilon}{\partial t}(x_1,t_n^{**})\Big)(t_1-t_n^*),
\end{eqnarray*}
where $t^{**}_n\in (t_1,t^*_n)$ and the last line comes again from the homogeneous Dirichlet boundary condition. As a consequence,  the function in the limit is bounded by $\frac{1}{n}\|u^\varepsilon\|_{W^{2,\infty}}$, which implies that the limit vanishes.

\medskip

 On the other hand, $\rho^\varepsilon=u^\varepsilon-u$ satisfies the following linear problem
$$
 \left\{ \begin{array}{ll}
  \partial_t \rho^\varepsilon -(1-\sigma)\Delta \rho^\varepsilon-\sigma(a-2bw)\Delta \rho^\varepsilon+2\sigma b\nabla(u^\varepsilon+u)\cdot\nabla\rho^\varepsilon&\\
  \hspace{7cm}-\sigma c\rho^\varepsilon +\sigma d (u^\varepsilon+u)\rho^\varepsilon=\varepsilon
 & \text{in } \Omega\times(0,T),  \\
         \mathcal{B}[\rho^\varepsilon]=0
 & \text{on } \partial\Omega \times(0,T),\\
        \rho^\varepsilon(x,0)=0,
        \end{array} \right.
$$
where all the coefficients are uniformly bounded in $\varepsilon$. Therefore by the maximum principle, we have
$$
\|u^\varepsilon-u\|_{L^\infty(\Omega\times[0,T])}=\|\rho^\varepsilon\|_{L^\infty(\Omega\times[0,T])}\leq C_T\,\varepsilon.
$$
By taking the limit $\varepsilon\rightarrow 0$ in $u^\varepsilon\geq 0$, we obtain that $u\geq 0$ in $\Omega\times[0,T]$.

Next we prove that $u\leq\frac{a}{2b}-\varepsilon_0$ in $\Omega\times[0,T]$.
Suppose that $\exists (x_0,t_0)\in {\overline{\Omega}}\times (0,T]$ such that
$$
\frac{a}{2b}-\varepsilon_0 < u(x_0,t_0)=\max_{\bar\Omega\times [0,T]} u(x,t).
$$
Then we have $\pa_t u(x_0,t_0) \ge 0$, 
and moreover $x_0 \in \Omega$ if we consider the Dirichlet boundary condition, so that
$$
0\leq \Big(\partial_t u -(1-\sigma)\Delta u-\sigma(a-2bw)\Delta u+2\sigma b|\nabla u|^2\Big)\Big|_{(x_0,t_0)}=\sigma u(c-du)|_{(x_0,t_0)},
$$
which means
$$
c-du(x_0,t_0)\geq 0\quad \Rightarrow\quad u(x_0,t_0)\leq \frac{c}{d}.
$$
In the case of the Neumann boundary conditions,  $x_0$ might appear on the boundary, but in this case, we still have $\nabla u(x_0,t_0)=0$, therefore the above argument also works.
This implies $\frac{a}{2b}-\varepsilon_0<\frac{c}{d}$, which is a contradiction with the assumption $\frac{a}{2b}-\varepsilon_0\geq \frac{c}{d}$. Therefore $\max_{\bar{\Omega}\times[0,T]}u(x,t)\leq\frac{a}{2b}-\varepsilon_0$.

Thus the map $\mathcal{T}:X\times[0,1]\rightarrow X$ is well defined. Due to the compact embedding from $C^{2+\alpha,1+\frac{\alpha}{2}}$ to $C^{\alpha,\frac{\alpha}{2}}$, we know that $\mathcal T(\cdot,\sigma): X\rightarrow X$ is a compact operator.

Next we show that the map $\mathcal T$ is continuous in $w$ and $\sigma$. For all $ w\in X$ and $\sigma\in[0,1]$, let $w_j\in X$ be a sequence such that $\|w_j-w\|_{C^{\alpha,\frac{\alpha}{2}}}\rightarrow 0$ as $j\rightarrow \infty$, and $\sigma_j\in [0,1]$ be a sequence such that $|\sigma_j-\sigma|\rightarrow 0$. Let $u_j=\mathcal{T}(w_j,\sigma_j)$, the Schauder estimates show that $\|u_j\|_{2+\alpha,1+\frac{\alpha}{2}}\leq C$ uniformly in $j$. Notice that $\rho_j=u_j-u$ satisfies the following linear problem
$$
\left\{  \begin{array}{ll}
  \partial_t \rho_j -(1-\sigma)\Delta \rho_j-\sigma(a-2bw)\Delta \rho_j+2\sigma b\nabla(u_j+u)\cdot\nabla\rho_j\\
  \hspace{5cm}+\sigma du\rho_j-\sigma (c-du_j)\rho_j=F_j
 & \text{in } \Omega\times(0,T),  \\
         \mathcal{B}[\rho_j]=0
 & \text{on } \partial\Omega \times(0,T),\\
        \rho_j(x,0)=0,
        \end{array} \right.
$$
where
\begin{eqnarray*}
F_j&=&(\sigma-\sigma_j)\Delta u_j-2\sigma b (w_j-w)\Delta u_j + (\sigma_j-\sigma)(a-2bw_j)\Delta u_j\\
&&\hspace{4cm}-2b(\sigma-\sigma_j)|\nabla u_j|^2 +(\sigma_j-\sigma) u_j (c-du_j).
\end{eqnarray*}
Using Schauder's theory for linear parabolic equations, we get the estimate
$$
\|u_j-u\|_{2+\alpha,1+\frac{\alpha}{2}}=\|\rho_j\|_{2+\alpha,1+\frac{\alpha}{2}}\leq C\|F_j\|_{\alpha,\frac{\alpha}{2}}\leq C(\|w_j-w\|_{\alpha,\frac{\alpha}{2}}+|\sigma_j-\sigma|).
$$
Hence, $\mathcal{T}$ is continuous in $w$ and $\sigma$.

Furthermore, it is obvious that $\mathcal T(w,0)=0$. Additionally, for any fixed point of $\mathcal T(u,\sigma)=u$, the uniform estimates for quasilinear parabolic equation show (\cite{LSU}, Chapter V, Theorem 7.2) that there exists a constant $M$ depending only on $\frac{a}{2b}$, $c$, $\|u_0\|_\infty$ such that
$$
\|\nabla u\|_{L^\infty},\|\partial_t u\|_{L^\infty}\leq M.
$$
Therefore, by Leray-Schauder's fixed point theorem, there exists a fixed point to the map $\mathcal{T}(\cdot,1)$, i.e. $u$ is a solution of the following problem:
\begin{align}
 \left\{ \begin{array}{ll}
  \partial_t u - (a - 2b\,u)\,\Delta u + 2b\,|\nabla u|^2
  = u(c-du)
 & \text{in } \Omega\times(0,T),  \\
         \mathcal{B}[u]=0
 & \text{on } \partial\Omega \times(0,T),\\
        u(x,0)=u_0(x),
        \end{array} \right.
\end{align}
which is equivalent to eq. (\ref{cdp0}).
\medskip

The uniqueness of classical solutions follows directly from comparison principles.
\end{proof}


\section{Steady states}
In this section, two results concerning  stationary states are given. One of them shows that nontrivial nonnegative solutions do not exist for eq. (\ref{cdp0}) with Dirichlet boundary condition. The other one has to do with the linear instability of constant steady states to eq. (\ref{cdp0}) with homogeneous Neumann boundary condition.

\subsection{Non-existence of non-trivial steady states for Dirichlet boundary conditions}

The steady states corresponding to \eqref{cdp0} satisfy the equation
\begin{align}\label{cdp0ssPoh}
  \left\{\begin{array}{ll}
- \Delta [(a-b u)u]=(c-d u)u
 & \text{in } \Omega,  \\
u=0
 & \text{on } \partial\Omega.
        \end{array}
\right.
\end{align}
We write down a non-existence result (based on Pohozaev method, see \cite{KuPo})  which works for general elliptic problems, and explain how to use it specifically in the case of eq. (\ref{cdp0ssPoh}).

\begin{thm}
Let  $n>2$, $\Omega$ be a star shaped domain of $\R^n$ with respect to the origin and suppose that $g,h$ are $C^1$ functions defined on $\R^+$ such that
\begin{equation}\label{genco}
0\leq n\int_0^sg(v)h'(v)\;dv
<
\frac{n-2}{2}g(s)h(s),\quad\forall s\geq0 .
\end{equation}
Then, the problem
\begin{align}\label{cdp0ssPohGen}
  \left\{\begin{array}{ll}
- \Delta h(u)=g(u)
 & \text{in } \Omega,  \\
       h(u)=0
 & \text{on } \partial\Omega,
        \end{array}
\right.
\end{align}
does not have any non-trivial (that is, different from $u \equiv 0$), nonnegative classical solutions.

As a consequence, considering $h(s) := (a-b s)\,s,\ g(s) := (c-d s)\,s$,  the sufficient condition (\ref{genco}) implies  non-existence of (non-trivial, nonnnegative, classical) solutions to eq. (\ref{cdp0ssPoh}) as soon as
$$
bds^2
+\left(
\frac{n-6}{6}ad-\frac{n+6}{6}bc
\right)s
+ac\leq0,\quad \forall s\geq0.
$$
This last condition is satisfied in particular when $c,d\leq0$ and
$
0<ad(n-6)-bc(n+6)\leq12\sqrt{abcd}
$
or
$
ad(n-6)-bc(n+6)\leq0
$. Note also that since $h(0)=0$, the homogeneous Dirichlet boundary condition $u=0$ on $\partial\Omega$ implies that $h(u)=0$
on $\partial\Omega$, so that Theorem 2 can be applied.
\end{thm}

\begin{proof}
By testing \eqref{cdp0ssPohGen} with $x\cdot\nabla h(u)$, we get
\begin{align}\label{poh0}
-\int_\Omega (x\cdot\nabla h(u)) \Delta h(u)\;dx
=
\int_\Omega (x\cdot\nabla h(u))g(u)\;dx.
\end{align}
First, notice that
\begin{align}
\nabla\cdot \left[(x\cdot\nabla h(u))\nabla h(u)\right]&=
(x\cdot\nabla h(u))\Delta h(u)
+\sum_k \left(\frac{\partial}{\partial x_k} [h(u)]\right)\frac{\partial}{\partial x_k}\left(\sum_ix_i\frac{\partial h(u)}{\partial x_i}\right)
\nonumber\\
&=(x\cdot\nabla h(u))\Delta h(u)
+|\nabla [h(u)]|^2
+\frac12\sum_i x_i\frac{\partial}{\partial x_i}|\nabla h(u)|^2.
\end{align}
Integrating over $\Omega$ and applying the divergence Lemma to the left hand side, we get
\begin{align}
	\int_{\partial \Omega }(x\cdot\nabla h(u))(\nabla h(u)\cdot\nu )\;d\sigma
&=
\int_\Omega (x\cdot\nabla h(u))\Delta h(u)\;dx
+\int_\Omega |\nabla h(u)|^2\;dx\nonumber\\
&+ \frac12\int_\Omega \sum_i x_i\frac{\partial}{\partial x_i}|\nabla h(u)|^2\;dx ,\nonumber
\end{align}
so that using \eqref{poh0},  we obtain
\begin{align}\label{poh1}
\int_{\partial \Omega }&(x\cdot\nabla h(u))(\nabla h(u)\cdot\nu )\;d\sigma
=\nonumber\\
&=-\int_\Omega (x\cdot\nabla h(u))g(u)\;dx
+\int_\Omega |\nabla h(u)|^2\;dx
+ \frac12\int_\Omega \sum_i x_i\frac{\partial}{\partial x_i}|\nabla h(u)|^2\;dx\\
&=: I+II+III.\nonumber
\end{align}
We first compute:
$$
I=-\int_\Omega (x\cdot\nabla h(u))g(u)\;dx
=
-\int_\Omega g(u)h'(u)\sum_ix_iu_{x_i}\;dx
$$
$$
=-\int_\Omega\sum_ix_i\frac{\partial F(u)}{\partial x_i}\;dx,
$$
where $F(u)=\int_0^ug(s)h'(s)\;ds$.
Integrating by parts, we get
$$
-\int_\Omega\sum_ix_i\frac{\partial F(u)}{\partial x_i}\;dx
=
n\int_\Omega F(u)\;dx
-\int_{\partial\Omega}(x\cdot\nu)F(u)\;dx,
$$
thus
\begin{equation}\label{PohI}
	I=
	n\int_\Omega F(u)\;dx
-\int_{\partial\Omega}(x\cdot\nu)F(u)\;d\sigma.
\end{equation}
For the second term, using problem \eqref{cdp0ssPohGen}, we get:
\begin{equation}\label{PohII}
II=\int_\Omega |\nabla h(u)|^2\;dx=\int_\Omega g(u)h(u)\;dx.
\end{equation}

For the last term in \eqref{poh1}, we compute
$$
\nabla \cdot (\frac{x}{2}|\nabla h(u)|^2)
=\frac{n}{2}|\nabla h(u)|^2
+\frac12\sum_{i} x_i\frac{\partial}{\partial x_i}|\nabla h(u)|^2,
$$
so that
\begin{eqnarray}
III
=\int_\Omega \frac12\sum_i x_i\frac{\partial}{\partial x_i}|\nabla h(u)|^2\;dx
&=&
\int_\Omega
\left(\mathrm{div}(\frac{x}{2}|\nabla h(u)|^2)-\frac{n}{2}|\nabla h(u)|^2\right)\;dx
\nonumber\\
&=&
\frac12\int_{\partial\Omega}x\cdot\nu|\nabla h(u)|^2\;d\sigma
-\frac{n}{2}\int_\Omega g(u)h(u)\;dx .
\label{PohIII}
\end{eqnarray}
Plugging \eqref{PohI},\eqref{PohII},\eqref{PohIII} into \eqref{poh1}, we obtain
\begin{align*}
\int_{\partial \Omega }(x\cdot\nabla h(u))(\nabla h(u)\cdot\nu )\;d\sigma
&=
 I+II+III
 \nonumber\\
&=
n\int_\Omega F(u)\;dx
-\int_{\partial\Omega}(x\cdot\nu)F(u)\;dx
+\int_\Omega g(u)h(u)\;dx
\nonumber\\
&+\frac12\int_{\partial\Omega}x\cdot\nu|\nabla h(u)|^2\;d\sigma
-\frac{n}{2}\int_\Omega g(u)h(u)\;dx .
\end{align*}
 Using the Dirichlet boundary condition, we see that $|\nabla h(u)|_{\partial\Omega}=|\nu\cdot\nabla h(u)|_{\partial\Omega}$, so that on $\partial\Omega$, we have $(x\cdot\nabla [h(u)])(\nu\cdot\nabla [h(u)])=(x\cdot\nu)|\nabla h(u)|^2$. Thus, the above relation becomes
\begin{align*}
\frac12\int_{\partial\Omega}(x\cdot\nu) |\nabla h(u)|^2 d\sigma&=
\int_{\partial \Omega }(x\cdot\nabla h(u))(\nabla h(u)\cdot\nu )d\sigma
-\frac12\int_{\partial\Omega} (x\cdot\nu) |\nabla h(u)|^2d\sigma=
\\
&=n\int_\Omega F(u)dx
-\int_{\partial\Omega}(x\cdot\nu)F(u)\;dx
-\frac{n-2}{2}\int_\Omega g(u)h(u), \;
\end{align*}
or
\begin{equation}
\frac12\int_{\partial\Omega}(x\cdot\nu) |\nabla h(u)|^2\;d\sigma
+\int_{\partial\Omega}(x\cdot\nu)F(u)\;dx
=
n\int_\Omega F(u)\;dx
-\frac{n-2}{2}\int_\Omega g(u)h(u)\;dx.
	\label{PohIdent}
\end{equation}

Since $\Omega$ is star shaped, there exists $\alpha\geq0$ such that
$$
x\cdot\nu\geq \alpha\int_{\partial\Omega}d\sigma\geq0,
$$
and relation \eqref{PohIdent} yields
\begin{equation*}
n\int_\Omega F(u)\;dx
-\frac{n-2}{2}\int_\Omega g(u)h(u)\;dx
\geq0.
\end{equation*}
Therefore, a sufficient condition for the non-existence of (non-trivial, nonnegative, classical) solutions is
\begin{equation}
n\int_0^sg(v)h'(v)\;dv
<
\frac{n-2}{2}g(s)h(s),\quad\forall s\geq0.
	\label{PohIneq}
\end{equation}

Next, if we set $h(s)=(a-b s)\,s,\ g(s)=(c-d s)\,s$,  then
$$ F(u)=
\frac{bd}{2}u^4-\frac{ad+2bc}{3}u^3+\frac{ac}{2}u^2, $$
$$ g(u)h(u)= bdu^4-(ad+bc)u^3+acu^2. $$
 Using \eqref{PohIneq}, we get the sufficient condition of non-existence of solutions to the corresponding steady-state problem, which consists in finding $a,b,c,d,n$ such that
$$
bd\, s^2
+\left(
\frac{n-6}{6}ad-\frac{n+6}{6}bc
\right)s
+ac\leq0,\quad \forall s\geq0.
$$
As stated in the theorem, this happens  when for example, $c,d\leq0$ and
$
ad(n-6)-bc(n+6)\leq12\sqrt{abcd}.
$

\end{proof}
\begin{rem}

We obtain \eqref{PohIneq} from \eqref{PohIdent} by neglecting the first boundary term (since $\Omega$ is starshaped). If now we keep this first boundary integral (the second boundary integral in \eqref{PohIdent} is 0 because of the boundary conditions) and compute:
\begin{align*}
\frac12\int_{\partial\Omega}(x\cdot\nu) |\nabla h(u)|^2\;d\sigma
&\geq
\frac{\alpha}{2}
\left(
\int_{\partial\Omega} -\frac{\partial h(u)}{\partial\nu}\;d\sigma
\right)^2
=
\frac{\alpha}{2}
\left(
\int_{\Omega} -\Delta h(u)\;dx
\right)^2\\
&=
\frac{\alpha}{2}
\left(
\int_{\Omega} g(u)\;dx
\right)^2 ,
\end{align*}
where we have used the geometry of the domain, Cauchy-Schwarz inequality, $|\partial\Omega|=1$, the divergence Lemma and the problem itself. Then,
relation \eqref{PohIdent} yields
\begin{equation*}
n\int_\Omega F(u)\;dx
-\frac{n-2}{2}\int_\Omega g(u)h(u)\;dx
\geq
\frac{\alpha}{2}
\left(
\int_{\Omega} g(u)\;dx
\right)^2.
\end{equation*}
Therefore, we can get a more precise description for the non-existence of solutions, since now we need to check the less stringent inequality:
\begin{equation}
n\int_\Omega F(u)\;dx
-\frac{n-2}{2}\int_\Omega g(u)h(u)\;dx
<
\frac{\alpha}{2}
\left(
\int_{\Omega} g(u)\;dx
\right)^2.
	\label{PohIneq2}
\end{equation}	
\end{rem}


\begin{rem}
We present here a computation related to the linear stability of steady states for the Neumann boundary condition.
We denote by $\{\lambda_k,e^k\}_{k=1}^\infty$  the solution of the eigenvalue problem for the Laplacian with homogeneous Neumann boundary condition,
with $\lambda_k\geq0$ for $k=1,2,\ldots$ and $0 =\lambda_1<\lambda_2\leq\lambda_3\leq\ldots$.
\medskip

We assume  that $c,d>0$. Then
the equilibrium $\frac{c}{d}$ for eq. (\ref{cdp0}) with homogeneous Neumann boundary condition is asymptotically linearly stable if and only if
 $ \frac{c}{d} \leq \frac{a}{2b}
  $.


Indeed we set
$u=\frac{c}{d}+\varepsilon w$, so that  the problem is transformed into
$$
\varepsilon\partial_t w
-\Delta \bigg[(a-b\frac{c}{d}-b\varepsilon w)(\frac{c}{d}+\varepsilon w) \bigg]
=
\bigg(c-d(\frac{c}{d}+\varepsilon w)\bigg)\,\bigg(\frac{c}{d}+\varepsilon w \bigg)
$$
$$ = - \var\, d\, w \bigg(\frac{c}{d}+\varepsilon w \bigg).$$
Thus
$$
\partial_t w
-\Delta \bigg[(a-\frac{2bc}{d}) w \bigg]
=
- cw
+O(\varepsilon).
$$
By projecting the equation onto the $k$-th eigenspace (and by using the notation $w_k(t) := \left<w(t,\cdot),e^k\right>$),
we obtain
$$
\frac{d}{dt} w^k
=
 \left((\frac{2bc}{d}-a)\lambda_k -c\right)w^k
+O(\varepsilon).
$$
The condition for linear asymptotic stability of the steady state $\frac{c}{d}$ is therefore, for all $k \in \N$,
$$
\left(\frac{2bc}{d}-a\right)\lambda_k-c<0,
$$
whence the result.



\end{rem}


\section{Blow up results}

In this section, we present blow up results (for different boundary conditions). Namely, we show that the solution to eq. (\ref{cdp0}) blows up, under appropriate conditions, for both Dirichlet and Neumann boundary conditions,
by using two different classical methods i.e. Kaplan's and concavity method. 

\subsection{A sufficient blow-up condition via Kaplan's method for Dirichlet boundary conditions}
The problem under consideration in this subsection is
\begin{align}\label{cdp0D}
  \left\{\begin{array}{ll}
  \partial_t u
  =  \Delta [(a-b u)u]+(c-d u)u
 & \text{in } \Omega\times(0,T),  \\
u=0,
 & \text{on } \partial\Omega \times(0,T),\\
u(x,0)=u_0(x)\geq 0.
        \end{array}
\right.
\end{align}

Let $(\mu,\phi)$ be the solution to the eigenvalue problem:
\begin{subequations}
	\label{Dir aux}
	\begin{equation}
 \; - \Delta\phi= \mu\phi,\quad x\in
\Omega,\label{Dir aux1}
	\end{equation}
	 	\begin{equation}
\phi=0,\quad x\in\partial\Omega ,
	 	\end{equation}
\end{subequations}
where $\mu$ is the first eigenvalue and $\Omega$ is a connected bounded domain.
Then
$\mu>0$ and $\phi$
is strictly positive and bounded in $\Omega$. For convenience, we also impose
the normalization condition
$\int_{\Omega}\phi(x)\, dx=1$. The main result in this subsection is
\begin{thm}\label{BUKaplan}
	Assume that $\Omega$ is a bounded smooth domain of $\R^n$ and let $u_0\in L^1(\Omega;\phi\, dx)$ satisfy
	$$
	A_0 := \int_\Omega u_0\,\phi>\frac{\max\{\mu a-c,0\}}{\mu b-d}\quad\text{and}\quad \mu b>d,
	$$
	where $(\mu,\phi)$ is the solution to \eqref{Dir aux}. Then any nonnegative  solution to  problem \eqref{cdp0D} blows up in finite time in $L^1(\Omega;\phi dx)$.
\end{thm}

\begin{rem}
When $c=d=0$, the above blow up condition on the initial data can be roughly translated as $ u_0 > Cst\,\frac{a}{b}$, which is coherent with our global existence result, and with the assumption of Theorem \ref{BUConcavity}. Note also that the homogeneous Dirichlet boundary condition $u=0$ on $\partial\Omega$ could be replaced in the theorem above by the less stringent condition $(a-bu)\,u=0$ on $\partial\Omega$.
\end{rem}

\begin{proof}

We begin, motivated by \cite{Kap63},  with testing (\ref{cdp0D}) with the eigenfunction $\phi$,
 and set $$A(t)=\int_{\Omega}\phi(x) \,u(x,t)\, dx ,$$
so that
\begin{align}
 A'(t)
  &=
  \int_\Omega \phi \Delta[(a-b u)u] dx
  +\int_\Omega (c-d u)u\phi dx
    \nonumber\\
      &=
  -\mu\int_\Omega (a-b u)u\phi dx
  +\int_\Omega (c-d u)u\phi dx
    \nonumber\\
  &=
 (c-\mu a)\int_\Omega u\phi
+(\mu b-d)\int_\Omega u^2\phi ,
\end{align}
where we have used problem \eqref{Dir aux}. Next we recall that
\begin{equation}\label{cond}
\mu b>d.
\end{equation}
After applying Jensen's  inequality, we get
$$
A'(t)\geq
 (c-\mu a)A(t)
+(\mu b-d)A^2(t),
$$
from which the blow up of the solution can be obtained. Namely, by using the change of variables
$ \Xi (t)=e^{-(c-\mu a)t}A(t) $, we can obtain
\begin{align*}
	 \Xi'(t)&=-(c-\mu a)e^{-(c-\mu a)t}A(t)+e^{-(c-\mu a)t}A'(t)
	 \\&\geq
 (\mu b-d)e^{-(c-\mu a)t}A^2(t)
 = (\mu b-d)e^{(c-\mu a)t}\Xi^2(t)
\end{align*}
or (as long as $\Xi(t)>0$, remembering that $\Xi(0)=A(0)>0$)
\begin{align}\label{Xi'}
	\left(\frac{-1}{\Xi}\right)'(t) \geq  (\mu b-d)e^{(c-\mu a)t}\Rightarrow
\frac{1}{\Xi(t)} \le \frac{1}{\Xi(0)}  - \frac{\mu\,b -d}{c - \mu\,a} \,	(e^{(c-\mu a)t}-1).
\end{align}
When $c - \mu\,a >0$, we see that $\Xi^{-1}(t) >0$ cannot remain true for
$$
 t \ge t^* := (c-\mu a)^{-1}\log\frac{(\mu b-d)A(0)+(c-\mu a)}{(\mu b-d) A(0)},
$$
so that blow up occurs before time $t^*$. When $c - \mu\,a <0$, a similar computation shows that a blow up also occurs, under the extra assumption $A_0 >\frac{\mu a-c}{\mu b-d}$.

%
\end{proof}


\subsection{A sufficient blowup condition by the concavity method for Neumann boundary condition}

As has been stated in the beginning of section 2, after the transformation $v=u-\frac{a}{2b}$, the equation can be rewritten into
$$
\partial_t v=-b\Delta v^2+(c-d\frac{a}{2b}-dv)(v+\frac{a}{2b}).
$$
In this subsection, we consider the following more general equation with homogeneous boundary condition,
\begin{align}
  \left\{\begin{array}{ll}
  \partial_t v
  =  -b\Delta v^m+h(v),
 & \text{in } \Omega\times(0,T),  \\
       \frac{ \partial v^m}{\partial \nu}=0,
 & \text{on } \partial\Omega \times(0,T),\\
v(x,0)=v_0(x)\geq0,& \text{in } \Omega,
        \end{array}
\right.\label{cdp1a}
\end{align}
and after giving a result about the blowup for the above general problem, we explain how (and under which conditions) it applies to problem \eqref{cdp1}. We refer the interested reader to \cite{Le90,Le00,Fi92,QuSou07}.
The main result of this subsection is the following

\begin{thm}\label{BUConcavity}
Suppose that $m>1$, and $h$ is a continuous real function such that for all $s\ge 0$, one has  $s^mh(s)\geq2 H(s)$,
where
$H(s) :=\int_0^s mt^{m-1}h(t)\;dt.$
We assume that $v:=v(x,t)$ is a smooth nonnegative solution to  problem \eqref{cdp1a} on $[0,T]$, such that
$$
\frac{b}{2}\int_\Omega |\nabla v_0(x)^{m}|^2\;dx
  +\int_\Omega H(v_0(x))\;dx>0. $$
Then there exists $t_*>0$ (depending only on $m$, $b$, $h$ and $v(0,\cdot)$) such that
 $T < t^*$.
 \medskip

  In other words, a blowup occurs before $t_*$. More precisely,
\begin{equation}\label{naa}
\lim_{t\to t^*}\int_0^t\int_\Omega v^{m+1}(x,\tau)\;dx\;d\tau=+\infty.
\end{equation}
\end{thm}

\begin{rem}
Note that the function $h$ is not assumed to be nonnegative. Actually, Theorem~\ref{BUConcavity} still holds when $h$ is negative, or when it changes sign.
\end{rem}

\begin{rem}
When we consider problem \eqref{cdp1} with homogeneous Neumann boundary condition, we are led to use $m=2$ and $ h(s):= (c-d\frac{a}{2b}-ds)(s+\frac{a}{2b})$ in Thm. \ref{BUConcavity}.  The condition $s^mh(s)\geq2 H(s)$ (for all $s\ge 0$)  becomes (remember that $a,b>0$, but the sign of $c,d$ is not fixed)
	\begin{align}\label{ccondition}
c\leq\min\left\{\frac{ad}{b},\frac{ad}{2b}\right\}.
	\end{align}
Coming back to the original unknown $u$ (instead of $v$), Thm.  \ref{BUConcavity} states that (under assumption (\ref{ccondition})) the (smooth) solutions to \eqref{cdp1} which are such that $u\geq \frac{a}{2b}$ is pointwise true, cannot exist globally.
 \medskip

 Note that a significant limitation of this result is related to the assumption that $u\geq \frac{a}{2b}$ pointwise. Indeed, this estimate is propagated at the formal level by the equation only in very special cases in which $a,b,c,d$ are linked by some equality, like when $c= ad/(2b) \ge 0$.
 
\end{rem}

\medskip

\begin{proof}
The proof is given by contradiction argument. Assume that the solution is global, and
define
$$
\Psi(t) := \int_0^t\int_\Omega v^{m+1}(x,\tau)\;dx\;d\tau \ge 0.
$$
The idea of the concavity method is to find an $\alpha>0$ and a $t_0\ge 0$ such that $\Psi^{-\alpha}$ is a concave function on $[t_0, +\infty[$.
Then, from the concavity property of $\Psi^{-\alpha}$ written at time $t_0$ in a differential way, we get
$$ \forall t \ge t_0, \qquad
\Psi^{-\alpha}(t)
\leq \Psi^{-\alpha}(t_0)
-\alpha \Psi^{-\alpha-1}(t_0)\Psi'(t_0)(t-t_0).
$$
Using this inequality together  with the fact that $\Psi'(t) \ge 0$ for all $t\ge 0$,
we obtain an upper bound $t^*$ for the blow up time (that is, the first  time $t^*$ such that $\Psi^{-\alpha}(t^*)=0$):
$$
t^*\leq\frac{\Psi^{-\alpha}(t_0)+\alpha t_0\Psi^{-\alpha-1}(t_0)\Psi'(t_0)}{\alpha \Psi^{-\alpha-1}(t_0)\Psi'(t_0)}
=
\frac{\Psi(t_0)+\alpha t_0\Psi'(t_0)}{\alpha \Psi'(t_0)}.
$$
To prove this concavity property, we compute
$$
(\Psi^{-\alpha})''
=
\alpha \Psi^{-\alpha-2}
\left(
(\alpha+1) \left(\Psi'\right)^2
-\Psi \Psi''
\right),
$$
from which it can be deduced that a sufficient condition for $\Psi^{-\alpha}$ to be concave (on $[t_0, +\infty[$) is
that $$ \forall t \ge t_0,\quad
\Psi(t) \,\Psi''(t)-(\alpha+1) \left(\Psi'(t)\right)^2 \ge 0.
$$

In fact, we start by computing the derivative of the functional:
\begin{align}
	\Psi'(t)&=\int_\Omega v^{m+1}(x,t)\;dx
=\int_0^t\int_\Omega (v^{m+1})_\tau\;dx\;d\tau
	+\int_\Omega v_0^{m+1}\;dx ,
\label{Psi'}
\end{align}
and its second derivative:
\begin{align}
	\Psi''(t)&=\int_\Omega( v^{m+1})_t\;dx
=(m+1)\int_\Omega v^{m}v_t\;dx
\label{Psi''} .
\end{align}
Next, we test \eqref{cdp1a} with $(m+1)v^{m}$, and get
\begin{align}
(m+1)  \int_\Omega v^{m}v_t\;dx
  &=  -b(m+1)\int_\Omega v^m\Delta v^m\;dx
  +(m+1)\int_\Omega v^mh(v)\;dx
  \nonumber\\
    &=
     b(m+1)\int_\Omega | \nabla v^{m}|^2\;dx
     +(m+1)\int_\Omega v^mh(v)\;dx\label{lmest} ,
\end{align}
so that substituting \eqref{lmest} into \eqref{Psi''}, we obtain
\begin{align}
\Psi''(t)=
  b(m+1)\int_\Omega | \nabla v^{m}|^2\;dx
     +(m+1)\int_\Omega v^mh(v)\;dx.\label{Psi''a}
\end{align}
We now test \eqref{cdp1a} with $(v^{m})_t$, and get
\begin{align}
 0\leq\int_\Omega mv^{m-1}(v_t)^2\;dx
  &=  -b\int_\Omega (v^{m})_t\Delta v^m\;dx
  +\int_\Omega (v^{m})_th(v)\;dx
  \nonumber\\
 &=\frac{d}{dt}\left\{ \frac{b}{2}\int_\Omega |\nabla v^{m}|^2\;dx
  +\int_\Omega H(v)\;dx\right\}
 =:
 \frac{d}{dt}E(t),
 \label{mEnergy}
\end{align}
where we recall that
$$
H(v) :=\int_0^v ms^{m-1}h(s)\;ds.
$$
From \eqref{mEnergy}, we can also deduce that
\begin{align}
	E(t)-E(0)&=\int_0^t\int_\Omega (v^m)_\tau v_\tau \;dx\;d\tau
	=m\int_0^t\int_\Omega v^{\frac{m-1}{2}} v_\tau v^{\frac{m-1}{2}} v_\tau \;dx\;d\tau
	\nonumber\\
	&=4\frac{m}{(m+1)^2}\int_0^t\int_\Omega [(v^{\frac{m+1}{2}})_\tau]^2 \;dx\;d\tau .
	\label{mEnergy2}
\end{align}
If the initial energy is strictly positive, namely
\begin{equation}
E(0)>0,
\label{E(0)pos}	
\end{equation}
 then we see that for all $t\ge 0$, $E(t)>0$, thanks to (\ref{mEnergy2}).

With the help of \eqref{mEnergy},
 identity \eqref{Psi''a} becomes
\begin{align}
\Psi''(t)&=
  b(m+1)\int_\Omega | \nabla v^{m}|^2
     +(m+1)\int_\Omega v^mh(v)\;dx
     \nonumber\\
     &= 2(m+1)E(t)
    +(m+1)\int_\Omega v^mh(v)\;dx
    -2(m+1)\int_\Omega H(v)\;dx .
\end{align}
At this point, we use the assumption on $h,H$ to obtain
\begin{equation}\label{gG-Condition}
s^mh(s)
\geq
2 H(s)
=
2\int_0^s mr^{m-1}h(r)\;dr
\quad \forall s\geq0 ,
\end{equation}
and  conclude that
\begin{align}
\forall t \ge 0, \qquad
\Psi''(t)\geq 2(m+1)E(t)\geq2(m+1)E(0)>0.
\label{Psi''final}
\end{align}
From the above inequality, we also get that $t \mapsto \Psi'(t)$ is strictly increasing.
Furthermore,
\begin{align}
\Psi''(t)\Psi(t)&\stackrel{\eqref{Psi''final}}{\geq}
 2(m+1)E(t)\Psi(t)
=
  2(m+1)E(t)  \int_0^t\int_\Omega (v^{\frac{m+1}{2}})^2 \;dx\;d\tau
 \nonumber\\
 &
\stackrel{\eqref{mEnergy2}}{=}
  2(m+1)\left(
4\frac{m}{(m+1)^2}\int_0^t\int_\Omega [(v^{\frac{m+1}{2}})_\tau]^2 \;dx\;d\tau
+E(0)\right)
 \int_0^t\int_\Omega (v^{\frac{m+1}{2}})^2\;dx\;d\tau
 \nonumber\\
 &\stackrel{\eqref{E(0)pos}}{>}
   \frac{8m}{m+1}
   \int_0^t\int_\Omega [(v^{\frac{m+1}{2}})_\tau]^2 \;dx\;d\tau
    \int_0^t\int_\Omega (v^{\frac{m+1}{2}})^2 \;dx\;d\tau
     \nonumber\\
 &\stackrel{Cauchy-Schwarz}{\geq}
    \frac{2m}{m+1}
  \left( \int_0^t\int_\Omega (v^{m+1})_\tau \;dx\;d\tau\right)^2
  =    \frac{2m}{m+1}\bigg(\Psi'(t)-\int_\Omega v_0^{m+1}\;dx \bigg)^2
  \nonumber\\
   &=    \frac{2m}{m+1}(\Psi'(t)-\Psi'(0))^2.
\end{align}
We now prove that there exists $\alpha>0$ and $t_0 >0$ such that
$$
\forall t \ge t_0, \qquad
 \frac{2m}{m+1}(\Psi'(t)-\Psi'(0))^2
 \geq
 (\alpha+1)\left(\Psi'(t)\right)^2 ,
$$
or equivalently
\begin{equation}\label{naq}
\forall t \ge t_0, \qquad
\left[
1-\left(\frac{(m+1)(\alpha+1)}{2m}\right)^\frac12
\right]
\Psi'(t)
\geq
\Psi'(0) .
\end{equation}
In order to do so, we choose $0<\alpha<\frac{m-1}{m+1}$ (remember that $m>1$). Due to the fact that $\Psi'(t)\to+\infty$ as $t\to+\infty$ (because of \eqref{Psi''final}), we can indeed choose $t_0>0$ large enough for (\ref{naq}) to hold.
\par
Therefore, we finally obtain
$$
\forall t \ge t_0, \qquad \Psi''(t)\Psi(t)\geq(\alpha+1)\left(\Psi'(t)\right)^2.
$$

As observed at the beginning of the proof,
the above inequality implies that we cannot extend the solution for all times, since (\ref{naa}) holds at a some point $t_*>0$.


\end{proof}


\section{Annex: Self similar solutions blowing up in the whole space}

In this annex, we provide a few explicit computations concerning  problem \eqref{cdp0} in the case when $c=d=0$:

$$
  \partial_t u
  =  \Delta [(a-b u)u],
$$
and we still consider only the nonnegative solutions.
\medskip

We recall that it is equivalent to studying the following problem:
 \begin{equation}\label{cdp1hom}
  \partial_t v+b\Delta v^2=0,
 \end{equation}
 with $u \ge 0 \iff  v\geq - \frac{a}{2b}$.
 \medskip

This equation is a (reverse in time) porous medium equation, for which
explicit solutions of  Barenblatt type \cite{Ba53} can be computed (on a given time interval
  $[0, T_*)$ for the first type given below):
\begin{equation}\label{BarSSS}
v(x,t)=\frac{1}{b\,(T_*-t)}
\left(
(T_*-t)^\frac{2}{n+2}-\frac{x^2}{4(n+2)}
\right)_+	,
\quad \text{in } \ \mathbb{R}^n\times [0,T_*) ,
\end{equation}
and
\begin{equation}\label{BarSSS2}
v(x,t)=\frac{-1}{b\,(T_* + t)}
\left(
(T_*+t)^\frac{2}{n+2}-\frac{x^2}{4(n+2)}
\right)_+	,
\quad \text{in } \ \mathbb{R}^n\times [0,+\infty).
\end{equation}
It is possible to take linear combinations of those solutions and still get solutions, though the equation is nonlinear, as long as the support of those solutions remain separate (more precisely, when each two solutions have support with empty intersections during the time of existence of the solutions).
\medskip

 For an inital datum
 $$ v_0(x) = \sum_{i=1}^A  \frac{1}{b\,T_i}
\left(
T_i^\frac{2}{n+2}-\frac{(x-x_i)^2}{4(n+2)}
\right)_+
- \sum_{j=A+1}^{A +B} \frac{1}{b\,T_j}
\left(
T_j^\frac{2}{n+2}-\frac{(x-x_j)^2}{4(n+2)}
\right)_+, $$
with $A, B \in \N$, $T_k>0$, $x_k \in \R^n$ ($k=1,\ldots,A+B$),
the function defined by
$$ v(t,x) = \sum_{i=1}^A  \frac{1}{b\,(T_i-t)}
\left(
(T_i-t)^\frac{2}{n+2}-\frac{(x-x_i)^2}{4(n+2)}
\right)_+
- \sum_{j=A+1}^{A +B} \frac{1}{b\,(t+T_j)}
\left(
(t+T_j)^\frac{2}{n+2}-\frac{(x-x_j)^2}{4(n+2)}
\right)_+, $$
is a solution to eq. (\ref{cdp1hom}) on the time interval $[0, \tau[$ for
$$ \tau := \min_{i=1,..,A} T_i, $$ (if $A=0$, $\tau =+\infty$)
provided  that
$$ \forall j,k = A+1,..,A+B, \,j \neq k, \qquad
 (\tau + T_j)^{\frac1{n+2}} + (\tau + T_k)^{\frac1{n+2}} < \frac{|x_k - x_j|}{2\sqrt{n+2}}, $$
 and
 \begin{equation}\label{la}
  \forall j = A+1,..,A+B, \, l = 1,..,A,\, t \in [0,\tau] \qquad
 (t + T_j)^{\frac1{n+2}} + (T_l - t)^{\frac1{n+2}} < \frac{|x_l - x_j|}{2\sqrt{n+2}}.
 \end{equation}
The condition (\ref{la}) can be rewritten without any direct reference to the time $t$
in the following way:
\begin{itemize}
\item
 If $T_l - T_j \le 0$,
$$ T_j^{\frac1{n+2}} + T_l^{\frac1{n+2}} < \frac{|x_l - x_j|}{2\sqrt{n+2}}, $$
\item
 If $T_l - T_j \ge 2\tau$,
$$ (\tau + T_j)^{\frac1{n+2}} + (T_l - \tau)^{\frac1{n+2}} < \frac{|x_l - x_j|}{2\sqrt{n+2}}, $$
\item
 If $T_l - T_j \in [0, 2\tau[$,
$$  2^{\frac{n+1}{n+2}}\, (T_l + T_j)^{\frac1{n+2}} < \frac{|x_l - x_j|}{2\sqrt{n+2}}. $$
\end{itemize}
Note also that $v(t,x) \ge - \frac{b}{2a}$ for all $t \in [0,\tau[$ and $x\in\R^n$
as soon as for all $j=A+1,..,A+B$, $T_j > (a/2)^{-1 - 2/n}$.
\medskip

The explicit solutions defined above feature in an explicit way the properties of blowup discussed previously. The value $v=0$ (or $u = \frac{a}{2b}$) plays a decisive role in the existence or not of a blowup, as can be guessed from the study of the parabolicity regions of the equation.
\medskip

Finally, we propose a figure illustrating the computations above.

\begin{figure}[hbt]\begin{center}
               \begin{minipage}{.5\textwidth}
                    \includegraphics[width=\textwidth]{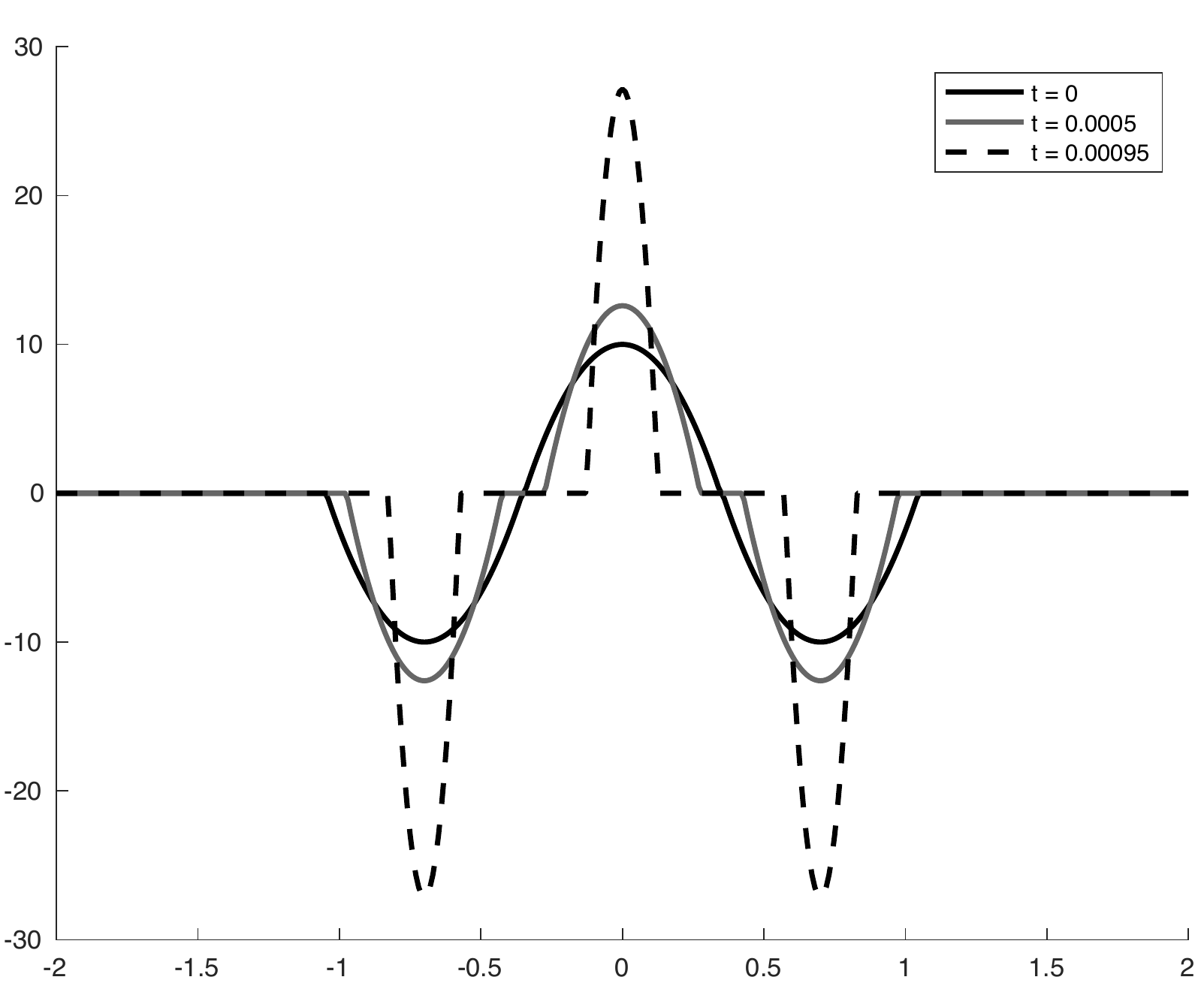}
                    \caption{Explicit solution shown at different times}
                \end{minipage}
               \end{center}
 \end{figure}

 In this figure, a solution is drawn, with one positive bump and two negative ones, with the specific feature that when $t=0$ the branches of the bump coincide and connect. For this solution, we drew three different time instances.
 \medskip

 \section{Conclusion}
  This paper is a first attempt to tackle problems of the form:
 \begin{align*}
  \left\{\begin{array}{ll}
  \partial_t u
  =  \Delta [(a-b u)u]+(c-d u)u
 & \text{in } \Omega\times(0,T),  \\
        \mathcal{B}[u]=0
 & \text{on } \partial\Omega \times(0,T),\\
u(x,0)=u_0(x)\geq 0,
        \end{array}
\right.
\end{align*}
whose main characteristic is the fact that the quantity inside the Laplacian does not a priori have a fixed sign, so that global-in-time existence of solutions does not always hold. We proved the global existence and uniqueness of classical solutions for initial data and parameters such that the problem is of parabolic type. The non-existence of non-trivial steady states is studied,
and some  blow up results using Kaplan's method on the one hand, and the concavity method on the other hand, are also presented.

\subsection*{Acknowledgements}
We thank Stephan Knapp for providing us with Figure 1.


\nocite{*}

  \bibliographystyle{alpha}
  \bibliography{CDPbiblio}

	



\end{document}